\documentclass{article}
\usepackage{amssymb}

\newtheorem{theorem}{Theorem}

\newtheorem{corollary}{Corollary}

\newtheorem{definition}{Definition}
\newtheorem{example}{Example}

\newtheorem{remark}{Remark}

\newenvironment{proof}[1][Proof]{\noindent\textbf{#1.} }{\ \rule{0.5em}{0.5em}}
\input{tcilatex}
\begin{document}

\title{{\Large Copula representations and order statistics for conditionally
independent random variables}}
\author{Ismihan Bairamov \\
%EndAName
Department of Mathematics, Izmir University of Economics, Izmir Turkey}
\date{ \ \ \ \ \ }
\maketitle

\begin{abstract}
The copula representations for conditionally independent random variables
and the distribution properties of order statistics of these random
variables are studied.

\end{abstract}

\section{Introduction}

Dawid (1979) noted that many of the important concepts of statistical theory
can be regarded as expressions of conditional independence, thus the
conditional independence offers a new language for the expression of
statistical concepts and a framework for their study. Dawid (1979)
considered \ the random \ variables $X\,,Y$ and $Z.$ If the random variables 
$X$ and $Y$ are independent in their joint distribution given $Z=z,$ for any
value of $z,$ then $X$ and $Y$ are called conditionally independent given $Z$
and denoted by ($X\perp Y\mid Z).$ A general calculus of conditional
independence is developed in \ Dawid (1980), where the general concept of
conditional independence for a statistical operation is introduced. Dawid
(1980) showed how the conditional independence for statistical operations
encompasses the basic properties such as, sufficiency, $(X\perp \Theta \mid
T),$ in which $X$ is a random variable with distributions governed by a
parameter $\Theta ,$ and $T$ is a function of $X,$ or$.$pointwise
independence, adequacy etc. \ Pearl et al. (1989) considered conditional
independence as a statement and addressed the problem of representing the
sum total of independence statements that logically follow from a given set
of such statements. \ Shaked and Spizzichino (1998) considered $n$
nonnegative random variables $T_{i},i=1,2,...,n$ which are interpreted as
lifetimes of $n$ units and assuming that $T_{1},T_{2},...,T_{n}$ are
conditionally independent given \ some random variable $\Theta ,$ determined
the conditions under which these conditionally independent random variables
are positive dependent. In the Bayesian setting it is of interest to know
which kind of dependence arises when $\Theta $ is unknown. Prakasa Rao
(2006) studied the properties of conditionally independent random variables
and proved the conditional versions of generalized Borel-Cantelli lemma,
generalized Kolmogorov's inequality, H\`{a}jek-R\'{e}nyi inequality.
Prakasa-Rao (2006) presented also the conditional versions of classical
strong law of large numbers and central limit theorem.

In this paper a different approach to conditionally independent random
variables is considered, the necessary and sufficient conditions for
conditional independence in terms of the partial derivatives of distribution
functions and copulas are given. Also, the distributional properties of
order statistics of conditionally independent random variables are studied.
The paper is organized as follows: in section 2 we present a definition of
conditionally independent random variables $X_{1},X_{2},....,X_{n}$ given $Z$
and derive a sufficient and necessary conditions for conditional
independence in terms of copulas. These conditions allow to construct
conditionally independent random variables with given bivariate
distributions of $(X_{i},Z)$ and marginal distributions of $X_{i}$'s. In
Section 3 we study the distributions of order statistics from conditionally
independent random variables. It is shown that these distributions can be
expressed in terms of partial derivatives of copulas of $X_{i}$ and $Z$. The
permanent expressions for distributions of order statistics are also
presented.

\section{Conditionally independent random variables}

Let $(X_{1},X_{2},...,X_{n},Z)$ be $n+1$ variate random vector with joint
distribution function (cdf) $%
H(x_{1},x_{2},...,x_{n},z)=C(F_{X_{1}}(x_{1}),F_{X_{2}}(x_{2}),...,F_{X_{n}}(x_{n}),F_{Z}(z)), 
$ where $F_{X_{i}}(x_{i})=P\{X_{i}\leq x_{i}\},$ $i=1,2,...,n$, $%
F_{Z}(z)=P\{Z\leq z\}$ and $C$ is a connecting copula.

\begin{definition}
If 
\begin{eqnarray}
P\{X_{1} &\leq &x_{1},X_{2}\leq x_{2},...,X_{n}\leq x_{n}\mid Z=z\}= 
\nonumber \\
P\{X_{1} &\leq &x_{1}\mid Z=z\}P\{X_{2}\leq x_{2}\mid Z=z\}\cdots
P\{X_{n}\leq x_{n}\mid Z=z\}  \label{1}
\end{eqnarray}%
for all $(x_{1},x_{2},...,x_{n},z)\in 
%TCIMACRO{\U{211d} }%
%BeginExpansion
\mathbb{R}
%EndExpansion
^{n+1},$ then the random variables $X_{1},X_{2},...,X_{n}$ are said to be
conditional independent given $Z.$
\end{definition}

It is clear that if the random variables $X_{1},X_{2},...,X_{n}$ are
conditionally independent given $Z$ \ then the conditional random variables $%
X_{i,z}\equiv (X_{i}\mid Z=z),i=1,2,...,n$ are independent for all $z\in 
%TCIMACRO{\U{211d} }%
%BeginExpansion
\mathbb{R}
%EndExpansion
$ and $P\{X_{1}\in B_{1},...,X_{n}\in B_{n}\mid \digamma _{Z}\}=P\{X_{1}\in
B_{1}\mid \digamma _{Z}\}\cdots P\{X_{n}\in B_{n}\mid \digamma _{Z}\}$ a.s.
for any Borel sets $B_{1},B_{2},...,B_{n}$, where $\digamma _{Z}$ is a $%
\sigma -$algebra generated by $Z.$

\begin{example}
\label{Example 0}Bairamov and Arnold (2008) discussed the residual
lifelengths of the remaining components in an $n-k+1-out-of-n$ system with
lifetimes of the components $X_{1},X_{2},\ldots ,X_{n}$, \ respectively. Let 
$X_{i}$' s are independent and identically distributed (iid) random
variables with common absolutely continuous distribution $F.$ If we are
given $X_{k:n}=x$, then the conditional distribution of the subsequent order
statistics $X_{k+1:n},\ldots ,X_{n:k}$ is the same as the distribution of
order statistics of a sample of size $n-k$ from the distribution $F$
truncated below at $x$. If we denote by $Y_{i}^{(k)},i=1,2,\ldots ,n-k$ the
randomly ordered values of $X_{k+1:n},\ldots ,X_{n:n}$, then given $X_{k:n}=x
$, these $Y_{i}^{(k)}$'s will be iid with common survival function $\bar{F}%
(x+y)/\bar{F}(x)$.
\end{example}

\bigskip\ \ 

\subsection{\protect\Large Assumptions and notations }

\textit{Throughout this paper we assume that }$X_{1},X_{2},...,X_{n},Z$%
\textit{\ are absolutely continuous random variables with corresponding
probability density functions (pdf) }$f_{X_{i}}(x),i=1,2,...,n$\textit{\ and 
}$f_{Z}(z),$\textit{\ respectively. Denote by }$%
F_{X_{1},X_{2},...,X_{n}}(x_{1},$\textit{\ }$x_{2},...,x_{n})$\textit{\ and }%
$C_{X_{1},X_{2},...,X_{n}}(u_{1},...,u_{n})$\textit{\ the joint distribution
function and the copula of }$(X_{1},X_{2},...,X_{n}),$\textit{\ respectively
and denote the pdf by }$f_{X_{1},X_{2},...,X_{n}}(x_{1},x_{2},...,x_{n}).$%
\textit{\ We use the following notations for the bivariate marginal
distributions and copulas of random variables }$X_{i}$\textit{\ and }$Z:$%
\textit{\ }$\ $\textit{\ }%
\[
F_{X_{i},Z}(x_{i},z)\equiv P\{X_{i}\leq x,Z\leq z\},i=1,2,...,n, 
\]%
\textit{and }%
\[
F_{X_{i},Z}(x_{i},z)\equiv C_{X_{i},Z}(F_{X_{i}}(x_{i}),F_{Z}(z)), 
\]%
\textit{where }$C_{X_{i},Z}(u,w),$\textit{\ }$(u,w)\in \lbrack 0,1]^{2},$%
\textit{\ is the copula of random variables }$(X_{i},Z).$\textit{\ We use
also the following notations for partial derivatives:}%
\[
\dot{H}(x_{1},x_{2},...,x_{n},z)=\frac{\partial H(x_{1},x_{2},...,x_{n},z)}{%
\partial z},\text{ }\dot{F}_{X_{i},Z}(x_{i},z)=\frac{\partial
F_{X_{i},Z}(x_{i},z)}{\partial z}, 
\]%
\[
\dot{C}(u_{1},u_{2},...,u_{n},w)\equiv \frac{\partial
C(u_{1},u_{2},...,u_{n},w)}{\partial w},\text{ \ }\dot{C}_{X_{i},Z}(u,w)%
\equiv \frac{\partial }{\partial w}C_{X_{i},Z}(u,w). 
\]

The following theorem gives a necessary and sufficient condition for
conditional independence of random variables $X_{1},X_{2},...,X_{n}$ given $%
Z.$

\subsection{Necessary and sufficient conditions for conditional independence}

\begin{theorem}
\label{Theorem 1}If $f_{Z}(z)>0,$ then the random variables $%
X_{1},X_{2},...,X_{n}$ are conditionally independent given $Z=z$ if and only
if 
\begin{equation}
\dot{H}(x_{1},x_{2},...,x_{n},z)=\left( \frac{1}{f_{Z}(z)}\right)
^{n-1}\dprod\limits_{i=1}^{n}\dot{F}_{X_{i},Z}(x_{i},z).  \label{2}
\end{equation}
\end{theorem}

\begin{proof}
It is clear that (\ref{1}) and (\ref{2}) both are equivalent to 
\begin{eqnarray*}
\lim\limits_{\Delta z\rightarrow 0}P\{X_{1} &\leq &x_{1},X_{2}\leq
x_{2},...,X_{n}\leq x_{n}\mid z\leq Z<z+\Delta z\} \\
&=&\lim\limits_{\Delta z\rightarrow 0}P\{X_{1}\leq x_{1}\mid z\leq
Z<z+\Delta z\}\cdots P\{X_{n}\leq x_{n}\mid z\leq Z<z+\Delta z\}.
\end{eqnarray*}%
\newline
\end{proof}

\begin{corollary}
\label{Corollary1}The random variables $X_{1},X_{2},...,X_{n}$ are
conditionally independent given $Z$ if and only if 
\begin{equation}
\dot{C}(u_{1},u_{2},...,u_{n},w)=\dprod\limits_{i=1}^{n}\dot{C}%
_{X_{i},Z}(u_{i},w)\text{ for all }0\leq u_{1},u_{2},...,u_{n},w\leq 1.
\label{3}
\end{equation}
\end{corollary}

\begin{proof}
From (\ref{2}) one can write%
\begin{eqnarray}
&&\frac{1}{f_{Z}(z)}\frac{\partial }{\partial z}\left[
C(F_{X_{1}}(x_{1}),F_{X_{2}}(x_{2}),...,F_{X_{n}}(x_{n}),F_{Z}(z))\right] 
\nonumber \\
&=&\left( \frac{1}{f_{Z}(z)}\right) ^{n}\dprod\limits_{i=1}^{n}\frac{%
\partial }{\partial z}\left[ C_{X_{i},Z}(F_{X_{i}}(x_{i}),F_{Z}(z))\right] .
\label{4}
\end{eqnarray}%
It follows from (\ref{4}) that 
\begin{eqnarray*}
&&\dot{C}(F_{X_{1}}(x_{1}),F_{X_{2}}(x_{2}),...,F_{X_{n}}(x_{n}),F_{Z}(z)) \\
&=&\dprod\limits_{i=1}^{n}\dot{C}_{X_{i},Z}(F_{X_{i}}(x_{i}),F_{Z}(z))
\end{eqnarray*}%
and the transformation $F_{X_{i}}(x_{i})=u_{i},(i=1,2,...,n),$ $F_{Z}(z)=w$
leads to (\ref{3}).
\end{proof}

\begin{corollary}
\label{Corollary2} $X_{1},X_{2},...,X_{n}$ are conditionally independent
given $Z,$ \ if and only if the following integral representations for joint
distribution function and copula hold true: 
\begin{eqnarray}
F_{X_{1},X_{2},...,X_{n}}(x_{1},x_{2},...,x_{n}) &=&\dint\limits_{-\infty
}^{\infty }\left( \frac{1}{f_{Z}(z)}\right) ^{n-1}\dprod\limits_{i=1}^{n}%
\frac{\partial F_{X_{i},Z}(x_{i},z)}{\partial z}dz  \nonumber \\
-\infty &<&x_{1}<\cdots <x_{n}<\infty  \label{4.1}
\end{eqnarray}%
and 
\begin{eqnarray}
C_{X_{1},X_{2},...,X_{n}}(u_{1},u_{2},...,u_{n})
&=&\dint\limits_{0}^{1}\dprod\limits_{i=1}^{n}\dot{C}_{X_{i},Z}(u_{i},w)dw, 
\nonumber \\
0 &\leq &u_{1},...,u_{n}\leq 1.  \label{4.2}
\end{eqnarray}
\end{corollary}

\begin{proof}
The proof can be made easily by integrating (\ref{2}) and (\ref{3}).
\end{proof}

\subsection{Construction of conditionally independent random variables}

Using Theorem 1 one can construct the conditionally independent random
variables $X_{1},X_{2},...,X_{n}$ given $Z,$ with given joint copulas $%
C_{X_{i},Z}(u_{i},w),$ $(u,w)\in \lbrack 0,1]^{2},$ of $(X_{i},Z),$ $%
i=1,2,...,n.$ The constructed joint distributions of conditionally
independent random variables can be used, for example, in reliability
analysis, for modeling lifetimes of the system having $n+1$ components, such
that knowing the exact lifetime of one of the components allows the
independence assumption for other components. Consider for example a system
of three dependent components with lifetimes $X_{1},X_{2},Z$ having joint
cdf $F_{X_{1},X_{2},Z}(x_{1},x_{2},z).$ Assume that the lifetime of the
system is $T=\max \{\min (Z,X_{1}),\min (Z,X_{2})\}$ and if $Z=z$ is given,
then the interaction between two other components becomes weaker, hence $%
X_{1}$ and $X_{2}$ can be assumed to be independent. In this and many
similar applications, where the conditional independence is a subject, the
constructed joint distributions of conditionally independent random
variables can be used for modeling of lifetimes or other random variables of
interest. In Shaked and Spizzichino \ (1998) some interesting applications
of conditionally independent random variables, such as an imperfect repair
with random effectiveness and lifetimes in random environments are discussed.

The following examples demonstrate the construction procedure of
conditionally independent random variables by using Theorem 1.

\begin{example}
\label{example1}Let $n=2.$ Denote by $C(u,v,w)$ the copula of random
variables \ $(X_{1},X_{2},Z)$, $\ $by $C_{X_{1},Z}(u,w)$ the copula of ($%
X_{1},Z)$ and by $C_{X_{2},Z}(v,w)$ the copula of $(X_{2},Z),$ respectively.
Assume that the joint distributions of ($X_{1},Z)$ and $(X_{2},Z)$ are
classic Farlie-Gumbel-Morgenstern (FGM) distribution, i.e. the copulas are 
\begin{eqnarray*}
C_{X_{1},Z}(u,w) &=&uw\{1+\alpha (1-u)(1-w)\},(u,w)\in \lbrack
0,1]^{2},-1\leq \alpha \leq 1 \\
C_{X_{2},Z}(v,w) &=&vw\{1+\alpha (1-v)(1-w)\},(v,w)\in \lbrack
0,1]^{2},-1\leq \alpha \leq 1,
\end{eqnarray*}%
with 
\begin{eqnarray*}
\dot{C}_{X_{1},Z}(u,w) &=&\frac{\partial }{\partial w}\left\{ uw\{1+\alpha
(1-u)(1-v)\}\right\} \\
&=&u+\alpha (1-u)(1-2w),(u,w)\in \lbrack 0,1]^{2},-1\leq \alpha \leq 1 \\
\dot{C}_{X_{2},Z}(v,w) &=&v+\alpha (1-v)(1-2w),(v,w)\in \lbrack
0,1]^{2},-1\leq \alpha \leq 1.
\end{eqnarray*}%
Then from equation (\ref{3}) we have 
\begin{eqnarray}
\dot{C}(u,v,w) &=&\dot{C}_{X_{1},Z}(u,w)\dot{C}_{X_{2},Z}(v,w)  \nonumber \\
&=&\left[ u+\alpha (1-u)(1-2w)\right] \left[ v+\alpha (1-v)(1-2w)\right] 
\nonumber \\
&=&uv+\alpha uv(2-u-v)(1-2w)+\alpha ^{2}uv(1-u)  \label{5} \\
&&\times (1-v)(1-2w)^{2}.  \nonumber
\end{eqnarray}%
Integrating (\ref{5}) with respect to $w,$ one obtains 
\begin{equation}
C(u,v,w)=uvw+\alpha uvw(2-u-v)(1-w)-\frac{\alpha ^{2}}{6}%
uv(1-u)(1-v)(1-2w)^{3}.  \label{6}
\end{equation}%
Therefore, the random variables $X_{1},X_{2}$ with the copula 
\[
C_{X_{1},X_{2}}(u,v)=uv+\frac{\alpha ^{2}}{6}uv(1-u)(1-v)\text{ } 
\]%
are conditionally independent if the copula of \ $(X_{1},X_{2},Z)$ is $%
C(u,v,w)$ given in (\ref{6}).
\end{example}

\begin{example}
\label{example2}Using calculations in Example 1 it is not difficult to see
that if the random variables $X_{1},X_{2},...,X_{n}$ and $Z$ are defined
with their copulas as%
\begin{eqnarray*}
C_{X_{1},Z}(u_{1},w) &=&u_{1}w\{1+\alpha (1-u_{1})(1-w)\}, \\
C_{X_{2},Z}(u_{2},w) &=&u_{2}w\{1+\alpha (1-u_{2})(1-w)\},-1\leq \alpha \leq
1
\end{eqnarray*}%
and 
\[
C_{X_{i},Z}(u_{i},w)=u_{i}w,\text{ }i=3,4,...,n,\text{ }0\leq
u_{1},u_{2},...,u_{n},w\leq 1, 
\]%
then the solution of the equation 
\[
\dot{C}(u_{1},u_{2},\cdots u_{n},w)=\dprod\limits_{i=1}^{n}\dot{C}%
_{X_{i},Z}(u_{i},w) 
\]%
is 
\begin{eqnarray}
C(u_{1},u_{2},...,u_{n},w) &=&u_{1}u_{2}\cdots u_{n}w+\alpha
u_{1}u_{2}\cdots u_{n}w(2-u_{1}-u_{2})(1-w)  \nonumber \\
&&-\frac{\alpha ^{2}}{6}u_{1}u_{2}\cdots u_{n}(1-u_{1})(1-u_{2})(1-2w)^{3}.
\label{5.1}
\end{eqnarray}%
The copula of $X_{1},X_{2},...,X_{n}$ is 
\begin{equation}
C_{X_{1},X_{2},...,X_{n}}(u_{1},u_{2},...,u_{n})=u_{1}u_{2}\cdots u_{n}+%
\frac{\alpha ^{2}}{6}u_{1}u_{2}\cdots u_{n}(1-u_{1})(1-u_{2})  \label{6.1}
\end{equation}%
and $X_{1},X_{2},...,X_{n}$ are conditionally independent given $Z.$
\end{example}

\begin{remark}
\label{remark 1}The conditional independence of random variables $\
X_{1},X_{2},...,X_{n}$ \ makes it possible to evaluate many important
probabilities concerning dependent random variables by replacing them with
the independent pairs of random variables. Assume that we need to calculate
the probability of some event connected with the dependent random variables $%
X_{1},X_{2},...,X_{n},$ for example consider 
\begin{eqnarray*}
P\{(X_{1},X_{2},...,X_{n}) &\in &\mathbf{B}\} \\
&=&\dint\limits_{\mathbf{B}%
}f_{X_{1},X_{2},...,X_{n}}(x_{1},x_{2},...,x_{n})dx_{1}dx_{2}\cdots dx_{n},
\end{eqnarray*}%
where $\mathbf{B}\in \Re ^{n},$ $\Re ^{n}$ is the Borel $\sigma -$algebra of
subsets of $%
%TCIMACRO{\U{211d} }%
%BeginExpansion
\mathbb{R}
%EndExpansion
^{n},$ $\mathbf{B}=B_{1}\times B_{2}\times \cdots \times B_{n}$ and $B_{i}$
are Borel sets on $%
%TCIMACRO{\U{211d} }%
%BeginExpansion
\mathbb{R}
%EndExpansion
$ and $f_{X_{1},X_{2},...,X_{n}}(x_{1},x_{2},...,x_{n})$ is the joint pdf of 
$X_{1},X_{2},...,X_{n}.$ If the random variables $X_{1},X_{2},...,X_{n}$ are
conditionally independent \ given $Z,$ \ then conditioning on $Z$ we have 
\begin{eqnarray}
P\{(X_{1},X_{2},...,X_{n}) &\in &B\}=\dint\limits_{-\infty }^{\infty
}P\{(X_{1},X_{2},...,X_{n})\in B\mid Z=z\}dF_{Z}(z)  \nonumber \\
&=&\dint\limits_{-\infty }^{\infty }\dprod\limits_{i=1}^{n}P\{X_{i}\in
B_{i}\mid Z=z\}dF_{Z}(z)=  \nonumber \\
&&\dint\limits_{-\infty }^{\infty }\left[ \frac{1}{f_{Z}(z)}\right]
^{n-1}\left\{ \dprod\limits_{i=1}^{n}\dint\limits_{B_{i}}\frac{\partial \dot{%
F}_{X_{i},Z}(x_{i},z)}{\partial x_{i}}dx_{i}\right\} dz.  \label{aa2}
\end{eqnarray}%
Furthermore, if $\mathbf{G}$ is some region on $%
%TCIMACRO{\U{211d} }%
%BeginExpansion
\mathbb{R}
%EndExpansion
^{n},$ then%
\begin{eqnarray*}
P\{(X_{1},X_{2},...,X_{n}) &\in &\mathbf{G}\}=\dint\limits_{-\infty
}^{\infty }P\{(X_{1},X_{2},...,X_{n})\in \mathbf{G}\mid Z=z\}dF_{Z}(z) \\
&=&\dint\limits_{-\infty }^{\infty }P\{(X_{1,z},X_{2,z},...,X_{n,z})\in 
\mathbf{G}\}dF_{Z}(z)
\end{eqnarray*}%
\begin{equation}
=\dint\limits_{-\infty }^{\infty }\left( \dint\limits_{\mathbf{G}%
}f_{X_{1,z}}(x_{1})f_{X_{2},z}(x_{2})\cdots
f_{X_{n},z}(x_{n})dx_{1}dx_{2}\cdots dx_{n}\right) dF_{Z}(z),  \label{aa4}
\end{equation}%
where $X_{i,z}=(X_{i}\mid Z=z),i=1,2,...,n$ are independent random variables
with pdf's 
\[
f_{X_{i,z}}(x_{i})=\frac{1}{f_{Z}(z)}\frac{\partial \dot{F}%
_{X_{i},Z}(x_{i},z)}{\partial x_{i}},i=1,2,...,n, 
\]%
respectively. \ In the following example we illustrate how to use (\ref{aa4}%
) in calculating of stress-strength probability $P\{X_{1}<X_{2}\}$ for
dependent, but conditionally independent random variables.
\end{remark}

\begin{example}
\label{example3}First we construct conditionally independent random
variables $X_{1},X_{2}$ given $Z=z.$ Let 
\begin{eqnarray*}
F_{X_{1},Z}(u,z) &=&u^{2}z\{1+(1-u^{2})(1-z)\},\text{ }%
F_{X_{1}}(u)=u^{2},F_{Z}(z)=z,\text{ }0\leq u,z\leq 1 \\
F_{X_{2},Z}(v,z) &=&vz\{1+(1-v)(1-z)\},\text{ }F_{X_{2}}(v)=v,F_{Z}(z)=z,%
\text{ }0\leq v,z\leq 1,
\end{eqnarray*}%
It is easy to calculate 
\begin{eqnarray}
\dot{F}_{X_{1},Z}(u,z) &=&u^{2}[1-(1-u^{2})(1-z)]+u^{2}z(1-u^{2})
\label{9.0} \\
\dot{F}_{X_{2},Z}(v,z) &=&v[1+(1-v)(1-z)]+vz(v-1).  \label{9}
\end{eqnarray}%
Then using Theorem 1 we have 
\begin{equation}
\dot{F}_{X_{1},X_{2},Z}(u,v,z)=\dot{F}_{X_{1},Z}(u,z)\dot{F}%
_{X_{2},Z}(v,z)dz.  \label{9.00}
\end{equation}%
Using (\ref{9.0}), (\ref{9} in \ref{9.00} ) and then integrating we have 
\begin{eqnarray}
&&F_{X_{1},X_{2},Z}(u,v,z)  \nonumber \\
&=&\frac{4}{3}u^{2}v^{2}z^{3}-\frac{4}{3}u^{2}vz^{3}+\frac{4}{3}u^{4}vz^{3}-%
\frac{4}{3}u^{4}v^{2}z^{3}-3z^{2}u^{4}v+2z^{2}u^{4}v^{2}+2z^{2}u^{2}v 
\nonumber \\
&&-z^{2}u^{2}v^{2}+2u^{4}vz-u^{4}v^{2}z.  \label{10}
\end{eqnarray}%
The joint cdf of $X_{1}$ and $X_{2}$ is 
\begin{equation}
F_{X_{1},X_{2}}(u,v)=\frac{2}{3}u^{2}v+\frac{1}{3}u^{2}v^{2}+\frac{1}{3}%
u^{4}v-\frac{1}{3}u^{4}v^{2}  \label{10.0}
\end{equation}%
and the pdf is 
\[
f_{X_{1},X_{2}}(u,v)=\frac{4}{3}u+\frac{4}{3}uv+\frac{4}{3}u^{3}-\frac{8}{3}%
u^{3}v. 
\]%
The corresponding copula of $(X_{1},X_{2})$ can be obtained by using
transformation $F_{X_{1}}(u)=u^{2}=t,$ $F_{X_{2}}(v)=v=s$ and it is 
\begin{eqnarray*}
C_{X_{1},X_{2}}(t,s) &=&\frac{2}{3}ts+\frac{1}{3}ts^{2}+\frac{1}{3}t^{2}s-%
\frac{1}{3}t^{2}s^{2}, \\
0 &\leq &t,s\leq 1.
\end{eqnarray*}
\end{example}

It follows from Theorem 1 that $X_{1},X_{2}$ are conditionally independent
given $Z=z.$

Now consider the probability $P\{X_{1}<X_{2}\}.$ By usual way integrating $%
f_{X_{1},X_{2}}(u,v)$ over the set $\{(u,v):u<v\}$ we obtain 
\begin{equation}
P\{X_{1}<X_{2}\}=\dint\limits_{0}^{1}\dint%
\limits_{0}^{v}f_{X_{1},X_{2}}(u,v)dudv=\frac{31}{90}.  \label{10.1}
\end{equation}%
On the other hand using conditional independence of $X_{1},X_{2}$ given $Z,$
we have 
\begin{eqnarray}
P\{X_{1} &<&X_{2}\}=\dint P\{X_{1}<X_{2}\mid Z=z\}dF_{Z}(z)  \nonumber \\
&=&\dint P\{X_{1,z}<X_{2,z}\}dF_{Z}(z)  \nonumber \\
&=&\dint\limits_{0}^{1}\left(
\dint\limits_{0}^{1}F_{1,z}(u)dF_{2,z}(u)\right) dF_{Z}(z),  \label{11}
\end{eqnarray}%
where $X_{1,z}\equiv X_{1}\mid Z=z$ and $X_{2,z}\equiv X_{2}\mid Z=z$ are
independent random variables with cdf's 
\begin{eqnarray*}
F_{X_{1,z}}(x) &=&\frac{1}{f_{Z}(z)}\dot{F}_{X_{1},Z}(x,z) \\
\text{ and }F_{X_{2,z}}(x) &=&\frac{1}{f_{Z}(z)}\dot{F}_{X_{2},Z}(x,z),
\end{eqnarray*}%
respectively. \ Since $f_{Z}(z)=1,$ then from (\ref{11}) taking into account
(\ref{9.0}) and (\ref{9}) we have 
\begin{eqnarray}
P\{X_{1} &<&X_{2}\}=\dint\limits_{0}^{1}\left( \dint\limits_{0}^{1}\left\{
u^{2}[1-(1-u^{2})(1-z)]+u^{2}z(1-u^{2})\right\} dF_{2,z}(u)\right) dz 
\nonumber \\
&=&\dint\limits_{0}^{1}\left( \dint\limits_{0}^{1}\left\{
u^{2}[1-(1-u^{2})(1-z)]+u^{2}z(1-u^{2})\right\} \left\{
1+(1-2u)(1-2z)\right\} du\right) dz  \nonumber \\
&=&\dint\limits_{0}^{1}\left( \frac{1}{15}+\frac{7}{15}z+\frac{2}{15}%
z^{2}\right) dz=\frac{31}{90},  \label{12}
\end{eqnarray}%
which agrees with (\ref{10.1}).

\section{Order statistics}

Denote by $X_{1:n},X_{2:n},...,X_{n:n}$ the order statistics of
conditionally independent random variables $X_{1},X_{2},...,X_{n}$ given $Z.$
We are interested in distribution of order statistics $X_{r:n},$ $1\leq
r\leq n$ and the joint distributions of $X_{r:n}$ and $X_{s:n},$ $1\leq
r<s\leq n.$ The formulas for distributions of order statistics contains
expressions depending on copulas of \ pairs $(X_{i},Z),$ $i=1,2,...,n$ and
permanents.

\subsection{Distribution of a single order statistics}

Conditioning on $Z$ one obtains 
\begin{eqnarray*}
P\{X_{n:n} &\leq &x\}=\dint\limits_{-\infty }^{\infty }P\{X_{n:n}\leq x\mid
Z=z\}dF_{Z}(z) \\
&=&\dint\limits_{-\infty }^{\infty }\dprod\limits_{i=1}^{n}P\{X_{i}\leq
x\mid Z=z\}dF_{Z}(z)=\dint\limits_{-\infty }^{\infty }\left( \frac{1}{%
f_{Z}(z)}\right) ^{n}\dprod\limits_{i=1}^{n}\frac{\partial
F_{X_{i},Z}(x_{i},z)}{\partial z}dF_{Z}(z) \\
&=&\dint\limits_{-\infty }^{\infty }\dprod\limits_{i=1}^{n}\dot{C}%
_{X_{i},Z}(F_{i}(x_{i}),F_{Z}(z))dF_{Z}(z)=\dint\limits_{0}^{1}\dprod%
\limits_{i=1}^{n}\dot{C}_{X_{i},Z}(F_{i}(x_{i}),s)ds
\end{eqnarray*}%
and similarly

\[
P\{X_{1:n}\leq x\}=1-\dint\limits_{0}^{1}\dprod\limits_{i=1}^{n}[1-\dot{C}%
_{X_{i},Z}(F_{i}(x_{i}),s)]ds. 
\]%
The distribution of $X_{r:n}$ ,$1\leq x\leq n$ can be derived as follows:%
\begin{eqnarray}
F_{r:n}(x) &\equiv &P\{X_{r:n}\leq
x\}=\dsum\limits_{i=r}^{n}\dint\limits_{-\infty }^{\infty }P\{\text{exactly }%
i\text{ of }X^{\text{'}}s\text{ are }\leq x\mid Z=z\}dF_{Z}(z)  \nonumber \\
&=&\dsum\limits_{i=r}^{n}\frac{1}{i!(n-i)!}\dsum%
\limits_{(j_{1},j_{2},...,j_{n})\in S_{n}}\dint\limits_{-\infty }^{\infty
}\left\{ \dprod\limits_{l=1}^{i}P\{X_{j_{l}}\leq x\mid Z=z\}\right. 
\nonumber \\
&&\left. \times \dprod\limits_{l=i+1}^{n}\left[ 1-P\{X_{j_{l}}\leq x\mid
Z=z\}\right] \right\} dF_{Z}(z)  \nonumber \\
&=&\dsum\limits_{i=r}^{n}\frac{1}{i!(n-i)!}\dsum%
\limits_{(j_{1},j_{2},...,j_{n})\in S_{n}}\dint\limits_{0}^{1}\left\{ \dot{C}%
_{X_{j_{l}},Z}(F_{X_{j_{l}}}(x),s)\right.  \label{6a} \\
&&\left. \times \dprod\limits_{l=i+1}^{n}\left[ 1-\dot{C}%
_{X_{j_{l}},Z}(F_{X_{j_{l}}}(x),s)\right] \right\} ds,  \nonumber
\end{eqnarray}%
where 
\begin{equation}
S_{n}=\{(j_{1},j_{2},...,j_{n}),1\leq j_{1},j_{2},...,j_{n}\leq n\},
\label{per}
\end{equation}%
is the set of all $n!$ permutations of $(1,2,...,n)$ and $%
\dsum\limits_{(j_{1},j_{2},...,j_{n})\in S_{n}}$ extends over all elements
of $S_{n}.$

\begin{remark}
\label{Remark 1}If $\ X_{1},X_{2},...,X_{n}$ are identically distributed,
i.e. $\ F_{X_{i}}(x)=F_{X}(x),\forall x\in 
%TCIMACRO{\U{211d} }%
%BeginExpansion
\mathbb{R}
%EndExpansion
,i=1,2,...,n$ and the joint distributions of ($%
X_{1},Z),(X_{2},Z),...,(X_{n},Z)$ are the same, i.e. $%
C_{X_{i},Z}(u,w)=C_{X,Z}(u,w),\forall (u,w)\in \lbrack 0,1]^{2},$ $%
i=1,2,...,n$ \ then 
\begin{equation}
F_{r:n}(x)=\dsum\limits_{i=r}^{n}\binom{n}{i}\dint\limits_{0}^{1}\left[ \dot{%
C}_{X,Z}(F_{X}(x),s)\right] ^{i}\left[ 1-\dot{C}_{X,Z}(F_{X}(x),s)\right]
^{n-i}ds.  \label{7}
\end{equation}
\end{remark}

\begin{remark}
\label{remark 1. 2}It follows from (\ref{7}) that, if $C_{X_{i},Z}(u,w)=uw,$
i.e. $X_{1},X_{2},...,X_{n}$ are independent and identically distributed \
(iid) random variables, then $\dot{C}_{X,Z}(F_{X}(x),s)=F_{X}(x)$ and%
\[
F_{r:n}(x)=\dsum\limits_{i=r}^{n}\binom{n}{i}F_{X}^{i}(x)\left[ 1-F_{X}(x)%
\right] ^{n-i}. 
\]
\end{remark}

\subsection{The joint distributions of two order statistics}

The joint distribution of $X_{r:n},$ and $X_{s:n}$ is%
\[
F_{r,s}(x,y)=P\{X_{r:n}\leq x,X_{s:n}\leq y\} 
\]%
\begin{eqnarray*}
&=&\dsum\limits_{i=r}^{n}\dsum\limits_{j=\max (s-i,0)}^{n-i}\frac{1}{%
i!(j-i)!(n-i-j)!}\dint\limits_{-\infty }^{\infty }\left\{
\dsum\limits_{(l_{1},...,l_{n})\in
S_{n}}\dprod\limits_{k=1}^{i}P\{X_{l_{k}}\leq x\mid Z=z\}\right. \\
&&\times \dprod\limits_{k=i+1}^{i+j}\left[ P\{X_{l_{k}}\leq y\mid
Z=z\}-P\{X_{l_{k}}\leq x\mid Z=z\}\right] \\
&&\left. \times \dprod\limits_{k=i+j+1}^{n}\left[ 1-P\{X_{l_{k}}\leq y\mid
Z=z\}\right] \right\} dF_{Z}(z)
\end{eqnarray*}%
and in terms of copulas 
\[
F_{r,s}(x,y) 
\]%
\begin{eqnarray}
&=&\dsum\limits_{i=r}^{n}\dsum\limits_{j=\max (s-i,0)}^{n-i}\frac{1}{%
i!(j-i)!(n-i-j)!}\dsum\limits_{(l_{1},...,l_{n})\in
S_{n}}\dint\limits_{0}^{1}\left\{ \dprod\limits_{k=1}^{i}\dot{C}%
_{X_{l_{k}},Z}(F_{X_{l_{k}}}(x),s)\right.  \nonumber \\
&&\times \dprod\limits_{k=i+1}^{i+j}\left[ \dot{C}%
_{X_{l_{k}},Z}(F_{X_{l_{k}}}(y),s)-\dot{C}_{X_{l_{k}},Z}(F_{l_{k}}(x),s)%
\right]  \nonumber \\
&&\left. \times \dprod\limits_{k=i+j+1}^{n}\left[ 1-\dot{C}%
_{X_{l_{k}},Z}(F_{X_{l_{k}}}(y),s)\right] \right\} ds  \label{8}
\end{eqnarray}

\begin{remark}
\label{Remark 2}If $\ X_{1},X_{2},...,X_{n}$ are identically distributed
with cdf $F_{X}(x),$ and the joint distributions of ($%
X_{1},Z),(X_{2},Z),...,(X_{n},Z)$ are the same, i.e. $%
C_{X_{i},Z}(u,w)=C_{X,Z}(u,w),\forall (u,w)\in \lbrack 0,1]^{2},$ $%
i=1,2,...,n,$ then%
\begin{eqnarray}
F_{r,s}(x,y) &=&\dsum\limits_{i=r}^{n}\dsum\limits_{j=\max (s-i,0)}^{n-i}%
\frac{n!}{i!(j-i)!(n-i-j)!}\dint\limits_{0}^{1}\left[ \dot{C}%
_{X,Z}(F_{X}(x),s)\right] ^{i}  \nonumber \\
&&\times \left[ \dot{C}_{X,Z}(F_{X}(y),s)-\dot{C}_{X_{l_{k}},Z}(F_{X}(x),s)%
\right] ^{j-i}  \label{8.1} \\
&&\times \left[ 1-\dot{C}_{X,Z}(F_{X}(y),s)\right] ^{n-i-j}ds  \nonumber
\end{eqnarray}
\end{remark}

\begin{remark}
\label{2.2}It follows from (\ref{8.1}) that, if $C_{X_{i},Z}(u,w)=uw,$ i.e. $%
X_{1},X_{2},...,X_{n}$ are iid\ random variables, then $\dot{C}%
_{X,Z}(F_{X}(x),s)=F_{X}(x)$ and%
\begin{eqnarray*}
F_{r,s}(x,y) &=&\dsum\limits_{i=r}^{n}\dsum\limits_{j=\max (s-i,0)}^{n-i}%
\frac{n!}{i!(j-i)!(n-i-j)!}F_{X}^{i}(x) \\
&&\times (F_{X}(y)-F_{X}(x))^{j-i}\left[ 1-F_{X}(y)\right] ^{n-i-j},
\end{eqnarray*}%
which agrees the well known formula for joint cdf of $r$th and $s$th order
statistics (see David and Nagaraja (2003))
\end{remark}

\subsection{Expressions for joint distributions of order statistics with
permanents}

Suppose $A=(a_{ij}),i,j=1,2,...,n$ is the square matrix. The permanent of $A$
is defined as 
\[
Per(A)=\sum_{(j_{1},j_{2},...,j_{n})\in
S_{n}}\dprod\limits_{k=1}^{n}a_{k,j_{k}}, 
\]%
where $S_{n}$ is defined in (\ref{per}) and $\sum_{(j_{1},j_{2},...,j_{n})%
\in S_{n}}$ denotes the sum over all $n!$permutations $%
(j_{1},j_{2},...,j_{n})$ of $(1,2,...,n).$ Using (\ref{6a}) one can realize
that 
\[
F_{r:n}(x)=\dsum\limits_{i=r}^{n}\frac{1}{i!(n-i)!}\dint\limits_{0}^{1}Per%
\left( M_{1}(x,s)\right) ds, 
\]%
where%
\begin{eqnarray*}
&&M_{1}(x,s) \\
&=&\left( 
\begin{tabular}{llll}
$\dot{C}_{X_{1},Z}(F_{X_{1}}(x),s)$ & $\dot{C}_{X_{2},Z}(F_{X_{2}}(x),s)$ & $%
\cdots $ & $\dot{C}_{X_{n},Z}(F_{X_{n}}(x),s)$ \\ 
$\dot{C}_{X_{1},Z}(F_{X_{1}}(x),x)$ & $\dot{C}_{X_{2},Z}(F_{X_{2}}(x),s)$ & $%
\cdots $ & $\dot{C}_{X_{n},Z}(F_{X_{n}}(x),s)$ \\ 
$\vdots $ & $\vdots $ &  & $\vdots $ \\ 
$\dot{C}_{X_{1},Z}(F_{X_{1}}(x),s)$ & $\dot{C}_{X_{2},Z}(F_{X_{2}}(x),s)$ & $%
\cdots $ & $\dot{C}_{X_{n},Z}(F_{X_{n}}(x),s)$ \\ 
&  &  &  \\ 
$1-\dot{C}_{X_{1},Z}(F_{X_{1}}(x),s)$ & $1-\dot{C}_{X_{2},Z}(F_{X_{2}}(x),s)$
& $\cdots $ & $1-\dot{C}_{X_{n},Z}(F_{X_{n}}(x),s)$ \\ 
$1-\dot{C}_{X_{1},Z}(F_{X_{1}}(x),s)$ & $1-\dot{C}_{X_{2},Z}(F_{X_{2}}(x),s)$
& $\cdots $ & $1-\dot{C}_{X_{n},Z}(F_{X_{n}}(x),s)$ \\ 
$\vdots $ & $\vdots $ &  & $\vdots $ \\ 
$1-\dot{C}_{X_{1},Z}(F_{X_{1}}(x),s)$ & $1-\dot{C}_{X_{2},Z}(F_{X_{2}}(x),s)$
& $\cdots $ & $1-\dot{C}_{X_{n},Z}(F_{X_{n}}(x),s)$%
\end{tabular}%
\right) 
\begin{tabular}{l}
$\left\} 
\begin{tabular}{l}
\\ 
$i$ \\ 
times \\ 
\\ 
\end{tabular}%
\right. $ \\ 
\\ 
$\left\} 
\begin{tabular}{l}
\\ 
$n-i$ \\ 
times \\ 
\\ 
\end{tabular}%
\right. $%
\end{tabular}%
.
\end{eqnarray*}%
Similar permanent expression for joint distribution of order statistics can
be obtained from (\ref{8})

\[
F_{r,s}(x,y)=\dsum\limits_{i=r}^{n}\dsum\limits_{j=\max (s-i,0)}^{n-i}\frac{1%
}{i!(j-i)!(n-i-j)!}\dint\limits_{0}^{1}M_{2}(x,y,s)ds, 
\]%
where%
\[
M_{2}(x,y,s)= 
\]%
\[
=\left( 
\begin{tabular}{llll}
$\dot{C}_{X_{1},Z}(F_{X_{1}}(x),s)$ & $\cdots $ & $\cdots $ & $\dot{C}%
_{X_{n},Z}(F_{X_{n}}(x),s)$ \\ 
&  &  &  \\ 
\begin{tabular}{l}
$\dot{C}_{X_{1},Z}(F_{X_{1}}(y),s)$ \\ 
$-\dot{C}_{X_{1},Z}(F_{X_{1}}(x),s)$%
\end{tabular}
& $\cdots $ & $\cdots $ & 
\begin{tabular}{l}
$\dot{C}_{X_{n},Z}(F_{X_{n}}(y),s)$ \\ 
$-\dot{C}_{X_{n},Z}(F_{X_{n}}(x),s)$%
\end{tabular}
\\ 
&  &  &  \\ 
$1-\dot{C}_{X_{1},Z}(F_{X_{1}}(x),s)$ & $\cdots $ & $\cdots $ & $1-\dot{C}%
_{X_{n},Z}(F_{X_{n}}(x),s)$%
\end{tabular}%
\right) 
\begin{tabular}{l}
$\left\} i\right. $ \\ 
\\ 
\\ 
$\left\} j-i\right. $ \\ 
\\ 
\\ 
$\left\} n-i-j\right. $%
\end{tabular}%
\]

In general the joint distribution function of order statistics $%
X_{r_{1}:n},X_{r_{2}:n},...,X_{r_{k}:n},$ $1\leq r_{1}<r_{2}<$ $\cdots
<r_{k}\leq n$ is 
\begin{eqnarray*}
&&F_{r_{1},r_{2}...,r_{k}}(x_{1},x_{2},...,x_{k})=\sum \frac{1}{%
j_{1}!j_{2}!\cdots j_{k}j_{k+1}!}\dint%
\limits_{0}^{1}M_{3}(x_{1},x_{2},...,x_{k},s)ds, \\
-\infty &<&x_{1}<x_{2}<\cdots <x_{k}<\infty ,
\end{eqnarray*}%
where the sum is over $j_{1},j_{2},...,j_{k+1}$ with $j_{1}\geq
r_{1},j_{1}+j_{2}\geq r_{2},...,j_{1}+j_{2}+\cdots +j_{k}\geq r_{k}$ and $%
j_{1}+j_{2}+\cdots +j_{k}+j_{k+1}=n$ and%
\begin{eqnarray*}
&&M_{3}(x_{1},x_{2},...,x_{k},s) \\
&=&\left( 
\begin{tabular}{llll}
$\dot{C}_{X_{1},Z}(F_{X_{1}}(x_{1}),s)$ & $\cdots $ & $\cdots $ & $\dot{C}%
_{X_{n},Z}(F_{X_{n}}(x_{1}),s)$ \\ 
&  &  &  \\ 
\begin{tabular}{l}
$\dot{C}_{X_{1},Z}(F_{X_{1}}(x_{2}),s)$ \\ 
$-\dot{C}_{X_{1},Z}(F_{X_{1}}(x_{1}),s)$%
\end{tabular}
& $\cdots $ & $\cdots $ & 
\begin{tabular}{l}
$\dot{C}_{X_{n},Z}(F_{X_{n}}(x_{2}),s)$ \\ 
$-\dot{C}_{X_{n},Z}(F_{X_{n}}(x_{1}),s)$%
\end{tabular}
\\ 
$\ \ \ \ \ \ \ \ \vdots $ &  &  & $\ \ \ \ \ \ \ \ \vdots $ \\ 
\begin{tabular}{l}
$\dot{C}_{X_{1},Z}(F_{X_{1}}(x_{k}),s)$ \\ 
$-\dot{C}_{X_{1},Z}(F_{X_{1}}(x_{k-1}),s)$%
\end{tabular}
& $\cdots $ & $\cdots $ & 
\begin{tabular}{l}
$\dot{C}_{X_{n},Z}(F_{X_{n}}(x_{k}),s)$ \\ 
$-\dot{C}_{X_{n},Z}(F_{X_{n}}(x_{k-1}),s)$%
\end{tabular}
\\ 
&  &  &  \\ 
$1-\dot{C}_{X_{1},Z}(F_{X_{1}}(x_{k}),s)$ & $\cdots $ & $\cdots $ & $1-\dot{C%
}_{X_{n},Z}(F_{X_{n}}(x_{k}),s)$%
\end{tabular}%
\right) 
\begin{tabular}{l}
$\left\} j_{1}\right. $ \\ 
\\ 
\\ 
$\left\} j_{2}\right. $ \\ 
\\ 
\\ 
$\left\} j_{k}\right. $ \\ 
\\ 
\\ 
$\left\} j_{k+1}\right. $%
\end{tabular}%
.
\end{eqnarray*}

\begin{remark}
\label{Remark 3}If $\ X_{1},X_{2},...,X_{n}$ are identically distributed
with cdf $F_{X}(x),$ and the joint distributions of ($%
X_{1},Z),(X_{2},Z),...,(X_{n},Z)$ are the same, i.e. $%
C_{X_{i},Z}(u,w)=C_{X,Z}(u,w),\forall (u,w)\in \lbrack 0,1]^{2},$ $%
i=1,2,...,n,$ then%
\begin{eqnarray*}
&&F_{r_{1},r_{2},...,r_{k}}(x_{1},x_{2},...,x_{n}) \\
&=&\sum \frac{1}{j_{1}!j_{2}!\cdots j_{k}j_{k+1}!}\dint\limits_{0}^{1}\left[ 
\dot{C}_{X,Z}(F_{X}(x_{1}),s)\right] ^{j_{1}}\left[ \dot{C}%
_{X,Z}(F_{X}(x_{2}),s)-\dot{C}_{X,Z}(F_{X}(x_{1}),s)\right] ^{j_{2}} \\
&&\times \left[ \dot{C}_{X,Z}(F_{X}(x_{k}),s)-\dot{C}_{X,Z}(F_{X}(x_{k-1}),s)%
\right] ^{j_{k}}\left[ 1-\dot{C}_{X,Z}(F_{X}(x_{k}),s)\right] ^{j^{k+1}}ds,
\\
-\infty &<&x_{1}<x_{2}<\cdots <x_{k}<\infty ,
\end{eqnarray*}%
where the sum is over $j_{1},j_{2},...,j_{k+1}$ with $j_{1}\geq
r_{1},j_{1}+j_{2}\geq r_{2},...,j_{1}+j_{2}+\cdots +j_{k}\geq r_{k}$ and $%
j_{1}+j_{2}+\cdots +j_{k}+j_{k+1}=n.$
\end{remark}

The order statistics are widely used in statistical theory of reliability.
The coherent system consisting of $n-$ components with lifetimes $%
X_{1},X_{2},...,X_{n}$ has lifetime which can be expressed with the order
statistics (or the linear functions of order statistics using Samaniego's
signatures). \ The $n-k+1$-$out$-$of$-$n$ coherent system, for example, has
lifetime $T=X_{k:n\text{ }}$ and the mean residual life function of such a
system at the system level is $\Psi (t)=E\{X_{k:n}-t\mid X_{r:n}>t\}$, $r<k.$
The function $\Psi (t)$ expresses the mean residual life (MRL) length of a $%
n-k+1$-$out$-$of$-$n$ system given that at least $n-r+1$ components are
alive at the moment $t.$ It is clear that to compute the reliability or the
MRL functions of such systems we need the joint distributions of order
statistics. The results presented in this paper can be used if the system
has components with conditionally independent lifetimes.

\end{document}